\documentclass[a4paper,11pt]{amsart}
\usepackage{enumerate}
\usepackage{amsfonts,amssymb}
\usepackage{float}
\usepackage{graphicx,tikz}
\usepackage{tikz-cd} 
\usepackage{float}
\usepackage[colorinlistoftodos,textsize=footnotesize]{todonotes}
\usepackage{amsmath, amsthm, amsfonts}
\usepackage{geometry}
\usepackage{hyperref}
\usepackage{enumitem}  
\usepackage[mathscr]{euscript}
\usepackage[capitalise]{cleveref}
\usepackage{comment}
\usepackage{enumitem}
\usepackage{amsrefs}
\usepackage[official]{eurosym}
\usepackage{mathtools}
\usepackage{graphicx}
\usepackage{nomencl}
\usepackage{amsfonts}
\usepackage{hyperref}
\usepackage{amssymb}
\usepackage{fancyhdr}
\usepackage{amscd}
\usepackage[english]{babel}
\usepackage{easybmat}
\usepackage{multirow,bigdelim}

\theoremstyle{plain}
\theoremstyle{definition}

\newtheorem*{theorem*}{Theorem}
\newtheorem{thm}{Theorem}
\newtheorem{pr}[thm]{Proposition}
\newtheorem{cor}[thm]{Corollary}
\newtheorem{lem}[thm]{Lemma}

\theoremstyle{definition}

\newtheorem{rem}[thm]{\scshape{Remark}}

\newcommand\T{{\mathcal{T}}}

\makeatletter
\def\cleardoublepage{\clearpage\if@twoside \ifodd\c@page\else
	\hbox{}
	\thispagestyle{empty}
	\newpage
	\if@twocolumn\hbox{}\newpage\fi\fi\fi}
\makeatother
\DeclareMathOperator{\Aut}{Aut}
\DeclareMathOperator{\End}{End}
\DeclareMathOperator{\st}{Stab}
\DeclareMathOperator{\St}{Stab}
\DeclareMathOperator{\Ker}{Ker}
\DeclareMathOperator{\Imm}{Im}

\DeclareMathOperator{\Sym}{Sym}
\DeclareMathOperator{\Alt}{Alt}

\keywords{Groups of automorphisms of rooted trees, verbal subgroups}
\subjclass[2020]{20E08}

\begin{document}
	
	\title[On verbal subgroups of $\Aut\T$]{On verbal subgroups of the group of automorphisms of regular rooted trees}
 \author[C. Delizia]{Costantino Delizia}
\address{Costantino Delizia: Dipartimento di Matematica, Università degli Studi di Salerno, Italy}
\email{cdelizia@unisa.it}
\author[M.\,E. Garciarena]{Mikel E. Garciarena}
 \address{Mikel E. Garciarena: Dipartimento di Matematica, Universit\`a di Salerno, 84084 Fisciano, Italy \--- Department of Mathematics, University of the Basque Country UPV/EHU, 48080 Bilbao, Spain}
 \email{mgarciarenaperez@unisa.it}
\author[M. Noce]{Marialaura Noce}
\address{Marialaura Noce: Dipartimento di Informatica, Università degli Studi di Salerno, Italy}
\email{mnoce@unisa.it}

\thanks{The second and the third authors are supported by the Spanish Government, grant PID2020-117281GB-I00, partly with FEDER funds.}	
	
\begin{abstract}
In this paper we revisit the description of all verbal subgroups of the group of automorphisms of a regular rooted tree $\T_d$, for $d>2$ and odd.
\end{abstract}

	\maketitle

\section{Introduction}

%\noindent Let $\T$ be the $d$-adic tree (that is, a regular rooted tree with $d$ descendants at every vertex), where $d\ge 2$. Then the automorphisms of $\T$ as a graph form a group $\Aut\T$ under composition; this is actually a profinite group.

\noindent Groups of automorphisms of regular rooted trees are a rich source of examples with interesting properties, and they have been used to solve very important problems  in Group Theory.
The first Grigorchuk group, defined by Grigorchuk \cite{Grig80} in 1980,
is one of the first instances of an infinite finitely generated periodic group, thus providing a negative solution to the General Burnside Problem.
It is also the first example of a group with intermediate growth \cite{Grig1}, hence solving the Milnor Problem, and the first group being amenable but not elementary amenable. 
Many other groups of automorphisms of rooted trees have since been defined and studied.
Important examples are the Gupta-Sidki $p$-groups~\cite{Gupta1983}, for $p$ an odd prime, and the second Grigorchuk group~\cite{Grig80}.
These are again finitely generated infinite periodic groups and they belong to the large family of the so-called \emph{Grigorchuk-Gupta-Sidki} groups.
Let $\T_d$, or simply $\T$ when it is clear from the context, be the $d$-adic tree (that is, a regular rooted tree with $d$ descendants at every vertex), where $d\ge 2$. We denote with $\Aut\T$ the set of all automorphisms of $\T$ as a graph. An automorphism of $\T$ fixes the root and preserves incidence. So, roughly speaking, it induces a permutation on each level of the tree.
The set of all automorphisms of $\T$ form a group under composition; this is actually a profinite group.
Remarkable subgroups of $\Aut\T$ have been studied in many different algebraic contexts, and also in dynamics and cryptography. Among others, for some classes of subgroups of $\Aut\T$ it has been studied the existence of maximal subgroups of finite index, the decidability of some algorithmic problems, and the description of their Schreier graphs. Some research has been also done regarding verbal subgroups and the study of the lower central series of certain subgroups of $\Aut\T$, but in general very little is known. 

We recall that if $w$ is a group-word, and $G$ is a group, then the verbal subgroup $w(G)$ of $G$ determined by $w$ is the subgroup generated by the set of all values $w(g_1, \dots, g_n)$, where $g_1, \dots, g_n$ run among the elements of $G$. It is well known that every verbal subgroup is fully invariant. Conversely, fully invariant subgroups need not be verbal.
The terms $\gamma_i(G)$ of the lower central series are examples of verbal subgroups. More precisely, $\gamma_i(G)$ is generated by the values in $G$ of the commutator
$[x_1,\ldots,x_i]$ of length $i$.
Other important examples of words related to commutators are the Engel word
$e_i=[x,y,\overset{i}{\ldots},y]$, and the derived words $\delta_i$ that define the terms of the derived series.%, and the outer commutator words, which are words obtained by nesting commutators but always using different indeterminates.For example, $[[x,y,z],[t,u]]$ is an outer commutator word.
%In \cite{Bier2021} Bier, Fern\'andez-Alcober, and Sushchanskyy have developed a method to calculate the terms of the lower central series of some finite $p$-groups that show an iterated structure with uniserial action. It applies in particular to finite iterated wreath products of the cyclic group of order $p$, reproving in a different way a classical result of Weir \cite{Weir1955}. This method can also provide information about the verbal subgroups corresponding to other words like the Engel words or outer commutators.

An automorphism $f$ is said to be \textit{finitary} if there exists a positive integer $m$ such that the permutation induced by $f$ is the identity for all levels $k \geq m$. The subgroup of $\Aut\T_d$ of all finitary automorphisms of $\T_d$ is denoted by $\mathcal{F}_d$. 
By using some techniques regarding tabular representations of groups and polynomial expressions, Smetaniuk and Sushchansky proved in \cite{verbald2} that every verbal subgroup of $\mathcal{F}_2$ coincides with some terms of the lower central series.
Further research has been done about the lower central series
$\{\gamma_i(G)\}_{i\ge 1}$ of specific subgroups $G$ of $\Aut\T$, and more specifically of some particular Grigorchuk-Gupta-Sidki groups, but also in this case the knowledge is really scarce.
If $G$ is the first Grigorchuk group then $|G:G'|=8$, and Rozhkov \cite{Rozhkov1996}
showed that $|\gamma_i(G):\gamma_{i+1}(G)|=2$ or $4$ for all $i\ge 2$.
On the other hand, Bartholdi \cite{Bartholdi2005} proved that
$|\gamma_i(G):\gamma_{i+1}(G)|$ is not bounded as $n\to\infty$ when $G$ is the Gupta-Sidki $3$-group. Bartholdi, Eick, and Hartung conjectured in \cite{Bartholdi2008} that
$|\gamma_i(G):\gamma_{i+1}(G)|\le p^2$ for the generalized Fabrykowski-Gupta groups.
A group $G$ in which $|\gamma_i(G):\gamma_{i+1}(G)|$ is bounded for all $i\ge 1$ is said to be of \emph{finite width}.
%If $G$ has finite width and is residually a finite $p$-group then the width of $G$ is defined to be $\max_{i\ge 1} \log_p |\gamma_i(G):\gamma_{i+1}(G)|$. In the book \cite{Klaas1997}, Klaas, Leedham-Green, and Plesken prove that a number of linear pro-$p$ groups related to classical groups over pro-$p$ domains have finite width.
Finally, Muntyan \cite{Muntyan} described the verbal subgroups of the automorphism group of a rooted tree of a regular rooted tree $\T_d$, where $d>2$ and odd. However this paper, which is available only in Russian, is difficult to comprehend since there are many results left unproven and some inaccuracies.

The goal of the present work is to fix this gap by providing an exhaustive and detailed version of the above results. % in order to  of the verbal subgroups of the group of automorphisms of a regular rooted tree $\T_d$, where $d>2$ and odd. 

The paper is organized as follows. In Section~2 we recall some general de\-fi\-nitions concerning the group of automorphisms acting on a rooted tree. In Section~3 we prove key preliminary results that we can be also used in a more general context, while in Section~4 we present the main  description of the verbal subgroups of the whole group of automorphisms of $\T_d$ for  $d>2$ and odd.

\section{Preliminaries}

\noindent A connected graph with no cycles is called a \textit{tree}. A tree is said to be \textit{rooted} it has a designated vertex called the \textit{root} and denoted by $\emptyset$. A rooted tree is \textit{regular} if every non-root vertex has the same \textit{valency}, defined as the number of edges passing through the vertex. If the valency of a regular rooted tree is $d+1$ then the \textit{degree} $d$ of the tree is the number of descendants of every non-root vertex. For every integer $d>1$ there is an unique regular rooted tree of degree $d$, that we will call simply the \textit{$d$-adic tree}, which looks like this:

%%%
\vspace{0.5cm}
\begin{center}
\includegraphics[scale = 0.46]{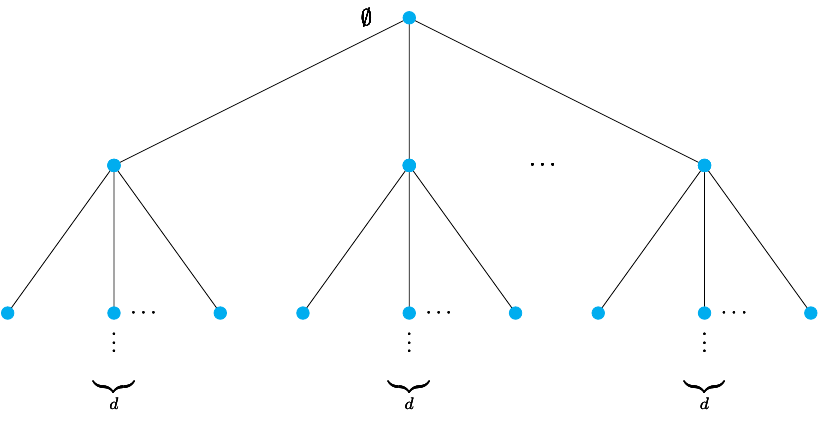}.
\end{center}
%%%

\noindent We denote this tree with $\T_d$. When the value of $d$ is clear from the context, we will simply write $\T$. Let $X = \{x_{1}, x_{2}, \dots, x_{d}\}$. A \textit{word} of positive length $n$ in $X$ is an expression of the form $x_{i_{1}} x_{i_{2}} \dots x_{i_{n}}$, where $i_{1}, i_{2}, \dots, i_{n} \in \{1, \dots, d\}$. Let $X^{n}$ be the set of all words of length $n$ and write $X^{*} = \bigcup_{n \ge0} X^{n}$. Then $X^{*}$ is the free monoid generated by $X$. Here the multiplication is defined by juxtaposition: if $u \in X^{n}$ and $v \in X^{m}$, then $uv \in X^{n+m}$. There is an obvious bijection between $X^{*}$ and $\T$: for all $n\ge0$, the elements in $X^n$ correspond to the $n$-th level of the tree.  Chosen a vertex $u$, the subtree of $\T$ hanging from $u$ constitutes the set $uX^{*}$, which is a copy of $\T$. It is customary to choose the set $X = \{1, \dots, d\}$.

The set $\Aut\T$ of all bijective maps of $X^{*}$ which preserve incidence, i.e.\ automorphisms of $\T$, is a group with respect to the ordinary composition between functions, defined by $(fg)(u) = g(f(u))$. Let $\alpha\in\Sym(d)$, the symmetric group on $d$ letters. Then the \textit{rooted automorphism} corresponding to $\alpha$ is the element of $\Aut\T$ obtained by rigidly permuting the subtrees hanging from the vertices $1, \dots, d$ using the permutation $\alpha$. On the other side, if $g\in\Aut\T$ is an automorphism sending $u$ to $v$ then
\begin{align*}
	u1 &\longmapsto v\alpha(1) \\
	u2 &\longmapsto v\alpha(2) \\
	&\,\,\,\vdots \\
	ud &\longmapsto v\alpha(d),
\end{align*}
where $\alpha$ is a suitable permutation. This element $\alpha \in \Sym(d)$ is called the \textit{label} of $g$ at the vertex $u$ and it is denoted with $g_{(u)}$. The \textit{portrait} of $g$ is the set of all labels of $g$.
Every $g\in\Aut\T$ is uniquely determined by its portrait. Indeed, if $u=x_1x_2\dots x_n\in X^n$, then we have
\begin{equation}\label{Eq:sequenceportrait}g(u)=g(x_1x_2\dots x_n)=(g_\emptyset(x_1)g_{(x_1)}(x_2)g_{(x_1x_2)}(x_3)\dots g_{(x_1x_2\dots x_{n-1})}(x_n)).
\end{equation}
The \textit{$n$-th level stabilizer} of $\Aut\T$ is 
$$
\St(n) = \{f \in \Aut\T \mid f(u) = u \ \forall u \in X^{n}\}.
$$
Clearly $\St(n)$ is a normal subgroup of $\Aut\T$ for all $n\in\mathbb{N}$. It immediately follows from (\ref{Eq:sequenceportrait}) that 
$$\St(1)\geq\St(2)\geq\dots\geq\St(n)\geq\dots.$$
We also have $$\bigcap_{n\in\mathbb{N}}\St(n)=\{1\}.$$ Let $f \in \Aut\T$ and $u \in X^{*}$, and suppose that $f(u) = u$. The \textit{section} of $f$ at $u$, denoted by $f_{u}$, is the restriction of $f$ to the subtree hanging from $u$, which can be identified with $\T$. Thus $f_{u}$ can be described by the rule $f(uv) = uf_{u}(v)$.  For all $n \in \mathbb{N}$ we can define an isomorphism
\begin{align*}
\psi_{n}: \St(n) & \longrightarrow \Aut\T \times \overset{d^n}{\dots} \times \Aut\T\\
g & \longmapsto (g_{u})_{u \in X^{n}}.
\end{align*}
Furthermore, straightforward computation shows that 
$$
\Aut\T=\Sym(d) \ltimes \St(1).
$$
This, together with the fact that $\St(1) \cong \Aut\T \times \overset{d}{\dots} \times \Aut\T$, implies %that $\Aut\T$ can also be seen as an iterated permutational wreath product: \begin{align*} \Aut\T & = \Sym(d) \ltimes \St(1) \cong \Aut\T \wr \Sym(d) \\ & \cong (\dots(\dots \Sym(d)\wr(\Sym(d) \wr \Sym(d)))).     \end{align*} From this it follows 
that every element $g \in \Aut\T$ can be uniquely written in the form 
\begin{equation}\label{Equation:expressionelements}
g=v \sigma, \mbox{ where } v=\psi_1^{-1}(g_1,\dots,g_d) \in \St(1), \mbox{ and } \sigma \in \Sym(d).
\end{equation}
Throughout the paper we will always write elements of $\Aut\T$ in the form \eqref{Equation:expressionelements} making the abuse of notation of omitting $\psi_1$, that is $g=(g_1, \dots, g_d)\sigma$. 

Let $g=(g_1, \dots, g_d)\sigma$ and $h=(h_1, \dots, h_d)\tau$ be elements of $\Aut\T$, where $(g_1, \dots, g_d), (h_1, \dots, h_d)\in\St(1)$ and $\sigma,\tau\in\Sym(d)$. Then a straightforward calculation shows that the following formulas hold: \begin{align}
\label{f1} gh&=(g_1h_{\sigma^{-1}(1)}, \dots, g_dh_{\sigma^{-1}(d)})\sigma\tau,\\
\label{f2}g^{-1}&=(g_{\sigma(1)}^{-1}, \dots, g_{\sigma(d)}^{-1})\sigma^{-1},\\
\label{f}g^\tau&=(g_{\tau(1)}, \dots, g_{\tau(d)})\sigma^\tau,\\
\label{f3}g^h&=(h_{\tau(1)}^{-1}g_{\tau(1)}h_{(\sigma^{-1}\tau)(1))}, \dots, h_{\tau(d)}^{-1}g_{\tau(d)}h_{(\sigma^{-1}\tau)(d))})\sigma^\tau,\\
\label{f4}[g,h]&=(g_{\sigma(1)}^{-1}h_{(\tau\sigma)(1)}^{-1}g_{(\tau\sigma)(1)}h_{(\tau^\sigma)(1)}, \dots, g_{\sigma(d)}^{-1}h_{(\tau\sigma)(d)}^{-1}g_{(\tau\sigma)(d)}h_{(\tau^\sigma)(d)})[\sigma,\tau].\end{align}
Let now $g=v\sigma$, where $v\in \St(1)$ and $\sigma \in \Sym(d)$. Then 
\begin{align*}
g^2&=v\sigma v\sigma=vv^{\sigma^{-1}}\sigma^2\\
g^3&=vv^{\sigma^{-1}}\sigma^2v\sigma=vv^{\sigma^{-1}}v^{\sigma^{-2}}\sigma^3
\end{align*}
and hence, for all positive integers $n$, we get
\begin{equation}\label{power}
g^n=vv^{\sigma^{-1}}v^{\sigma^{-2}}\dots v^{\sigma^{-(n-1)}}\sigma^n.
\end{equation}
A subgroup $G$ of $\Aut\T$ is said to be \emph{spherically transitive} (or \emph{level transitive}) if it acts transitively on each level of the tree. That is, if for all $n\in\mathbb{N}$ and $u,v\in X^n$, there exists $g\in G$ such that $g(u)=v$. Of course, the
whole group $\Aut\T$ is spherically transitive.

%If $p$ be a prime and $\sigma = (1 \dots p)$, we define the \emph{Sylow pro-$p$ subgroup} of $\Aut\T_p$ as the subgroup $\Gamma\leq\Aut\T_p$ that is mapped isomorphically to $\langle\sigma \rangle \wr \langle\sigma \rangle \wr\dots = C_p \wr C_p \wr\dots$. In this case, every element $g \in G \leq \Gamma$ can be written in the form $g = h\sigma^t$, for some $t \in \mathbb{Z}$ and $h \in \St(1)$ such that $\psi(h) \in G \times \dots \times G$.

 \section{First results}
%In this section, we denote with $\Aut\T_n$ (or $\Aut\T$ when we don't have restrictions on $n$) the group of automorphisms of a $n$-adic tree.
%Description of the commutative. If an automorphism $g$ of $G_n$ is represented as (*), then for each $g_i$ there are $g_{i1},\dots, g_{in}\in G_n$ and $\pi_1,\dots, \pi_n\in S_n$ such that
%\begin{align*}
 %   g_i=(g_{i1,\dots, g_{in}})\pi_i.
%\end{align*}
%Continuing this process indefinitely, we obtain a description of elements $G_n$ by means of their portraits - trees $T_n$ whose vertices are marked by permutations of $S_n$ (see [1]). For example, the group $\St(k)$ consists exactly of those elements $G_n$ in whose portraits the root vertex and vertices of the first $k$ levels are marked by an identical permutation; in the portrait of a single element of the group $G_n$ all vertices are marked by an identical permutation.

\noindent In this section we prove some preliminary lemmas that are crucial for the study of verbal subgroups of $\Aut\T$. 
\begin{lem}[\cite{Conjugate_of_its_inverse}]\label{Lemma: conjucagyinverse}
Every element of $\Aut\T$ is conjugate to its inverse. 
\end{lem}

In other words,  for every $g\in \Aut\T$, there exists $x\in \Aut~\T$ such that $ g^{-1}=g^x$. Let $\T$ be the $d$-adic tree. The automorphism $$t=(1,\dots, 1, t)\sigma, \text { with } \sigma=(1 \  2 \dots \ d),$$ is called the \textit{adding machine} on $\T$.

\begin{lem}[\cite{Cyclic_renormalization}]\label{Lemma: Cyclic_renormalization}
An automorphism $a \in \Aut\T$ is spherically transitive if and only if it is conjugate to the adding machine.
\end{lem}

\begin{lem}\label{Lemma: Product_of_spherically_transitive}
Let $g \in \Aut\T$, and assume that $g=uu'$, where both $u$ and $u'$ are spherically transitive. Then $g$ is a commutator of the form $[t^a,b]$, for some $a,b\in\Aut\T$. In particular, $g \in (\Aut\T)'$.
\end{lem}
\begin{proof}
By Lemma \ref{Lemma: Cyclic_renormalization} we can write $u=t^x$ and $u'=t^y$, where $t$ is the adding machine and $x,y\in \Aut\T$. By Lemma \ref{Lemma: conjucagyinverse}, there exists some $z\in \Aut\T$ such that $(t^{-1})^z=t$, and thus
\begin{align*}
    g&=uu'=t^xt^y=((t^{-1})^z)^xt^y=(t^{-1}(t^{y})^{(zx)^{-1}})^{zx}\\
    &=[t,y(zx)^{-1}]^{zx}=[t^{zx},(y(zx)^{-1})^{zx}]=[t^{zx},x^{-1}z^{-1}y],
\end{align*}
and the result follows.
\end{proof}

\begin{lem}\label{Lemma: Spherically transitive}
Let $a \in \Aut\T$ defined by:
$$
a = (1, \dots, 1, a_1)\sigma_1, \qquad a_k = (1, \dots, 1, a_{k+1})\sigma_{k+1}\quad \mbox{ for } k \geq 1,
$$
where each $\sigma_k$ is a cycle of length $d$. Then $a$ is spherically transitive.
\end{lem}
\begin{proof}
Since each $\sigma_k$ is a cycle of length $d$, it is conjugate in $\Sym(d)$ to the cycle $\sigma=(1 \ 2 \dots \ d)$. Thus for each
$k\geq 1$ there exists a permutation $\alpha_k$ such that $\alpha_k(d)=d$ and ${\sigma_k}^{\alpha_k}=\sigma$. Define $x\in \Aut \T$ as follows:
\begin{align*}
    x=(x_1,\dots,x_1)\alpha_1, \qquad x_k=(x_{k+1},\dots, x_{k+1})\alpha_{k+1},\quad k\geq 1.
\end{align*}
Then by (\ref{f3}) we have
\begin{align*}
    a^x=(1,\dots, 1,a_1^{x_1})\sigma,\qquad a_k^{x_k}=(1,\dots, 1,a_{k+1}^{x_{k+1}})\sigma.
\end{align*}
Since every element of $\Aut \T$ is uniquely determined by its portrait, it follows that $a^x=t$.
\end{proof}

%\begin{lem}    Any even permutation on $n$ points is represented as the product of two cycles of length $n$.
%\end{lem} \textcolor{blue}{To add in the preliminaries (I think this is known): any even permutation on $n$ points is represented as the product of two cycles of length $n$.}

Let $\T$ be the $d$-adic tree, and denote by $\Alt(d)$ the alternating group on $d$ letters. Consider the map $$P:g\in\Aut\T\mapsto(\varepsilon_0(g), \varepsilon_1(g), \varepsilon_2(g),\dots)\in\mathbb Z_2^{\mathbb N_0},$$
where
\begin{align*}
    \varepsilon_0(g)=\begin{cases}0 \quad&\text{if } g_\emptyset\in\Alt(d),\\
    1\quad&\text{if } g_\emptyset\not\in\Alt(d),\end{cases}
    \end{align*}
and, for all $n\in\mathbb N,$
$$
    \varepsilon_n(g)=\begin{cases}0 \quad&\text{if } \displaystyle\prod_{u\in X^n} g_{(u)}\in\Alt(d),\\ \\
    1\quad&\text{if } \displaystyle\prod_{u\in X^n}g_{(u)}\not\in\Alt(d).\end{cases}
$$
Then it is clear by (\ref{f1}) that $P$ is an epimorphism and so the image of $P$ is an abelian group. Therefore  $$(\Aut\T)'\leq\ker P=\{g\in\Aut\T\,|\,P(g)=(0,0,0,\dots)\}.$$ Notice that $\ker P$ is actually the set of those automorphisms of $\T$ for which the product of all elements in each level of the portrait is an even permutation. Hence recursively we have
$$
\ker P = \{(g_1, \dots, g_d)\sigma \mid g_1\dots g_d \in \ker P,\, \sigma \in \Alt(d)\}.
$$
%\textcolor{magenta}{
%Let $g$ be an automorphisms identified with its sequence of portrait  $[a_1,a_2(x_1),a_3(x_1,x_2),\dots]$ as in \eqref{Eq:sequenceportrait}. Then  $$ \ker P = \{(g_1, \dots, g_d)\sigma \mid g_1\dots g_d \in \ker P, \sigma \in \Alt(d)\}, $$ where $\Alt(d)$ denotes the alternating group on $d$ letters.}

\begin{thm}\label{Theorem: (AutT)'=ker(P)}
The derived subgroup $(\Aut\T)'$ is equal to $\ker P$. Moreover, every element of $(\Aut\T)'$ can be written as a product of two spherically transitive automorphisms. In particular, every element of $(\Aut\T)'$ is a commutator.
\end{thm}
\begin{proof}
We already noticed that $(\Aut\T)'\leq \ker P$. %Indeed, the map \begin{align*}
  %  P:\Aut~\T\longrightarrow \prod_{i=1} \mathbb{Z}/2\mathbb{Z}
%\end{align*}
%is a group epimorphism, so $\Aut~\T/\ker P\cong \prod_{i=1} \mathbb{Z}/2\mathbb{Z}$, which implies that  $\Aut~\T/\ker P$ is abelian and thus $(\Aut~\T)'\leq \ker P$.\\
By Lemma \ref{Lemma: Product_of_spherically_transitive} we only need to show that every element in $\ker P$ is a product of two spherically transitive automorphisms. 

Let $g=(g_1,\dots, g_d)\sigma\in \ker P$, and so we have that $g_1\dots g_d\in \ker P$ and $\sigma \in \Alt(d)$. Our goal is to recursively construct $u$ and $y$ in $\Aut\T$ such that $g=uu^y$ and $u$ also satisfies \cref{Lemma: Spherically transitive}, which ensures that $u$ is spherically transitive. We start by setting
\begin{align*}
   u&=(1,\dots, 1,u_1)\alpha,\\
   y&=(x_1,\dots, x_d)\beta,
\end{align*}
where $\alpha,\beta\in \Sym(d)$ and $\alpha$ is a cycle of length $d$. Hence  $\alpha^{-1}=(i_1\ i_2\ \dots\ i_d)$ where $\{i_1,i_2,\dots,i_{d-1}\}=\{1,2,\dots,d-1\}$, $i_d=d$, and $\alpha\alpha^{\beta}=\sigma$. Set $$k_i=(\beta\alpha^{-1})(i)$$ for all $i=1,\dots,d$. So we have
$(\alpha^{-1}\beta\alpha^{-1})(i)=(\beta\alpha^{-1})(\alpha^{-1}(i))=k_{\alpha^{-1}(i)}$, and hence
\begin{align*}
    (\alpha^{-1}\beta\alpha^{-1})(i)=\begin{cases}k_{i_{j+1}} \quad&\text{if $i=i_j<d$,}\\
    k_{i_1}\quad&\text{if $i=d$.}\end{cases}
\end{align*}
By direct computation, using (\ref{f1}), (\ref{f2}) and (\ref{f3}) we obtain 
$$
    uu^y=\big(x^{-1}_{k_1}u_1^{\delta(k_1)}x_{k_{\alpha^{-1}(1)}}\,,\,\dots\,,\,x^{-1}_{k_{d-1}}u_1^{\delta(k_{d-1})}x_{k_{\alpha^{-1}(d-1)}}\,,\,u_1x^{-1}_{k_d}u_1^{\delta(k_d)}x_{k_{\alpha^{-1}(d)}}\big),
$$
where
\begin{align*}
    \delta(k)=\begin{cases}0 \quad&\text{if $k<d$,}\\
    1\quad&\text{if $k=d$.}\end{cases}
\end{align*}
The required equality $g=uu^y$ yields
\begin{align*}
    \begin{cases}
        x^{-1}_{k_1}u_1^{\delta(k_1)}x_{k_{\alpha^{-1}(1)}}&=g_1,\\
        &\,\,\vdots\\
        x^{-1}_{k_{d-1}}u_1^{\delta(k_{d-1})}x_{k_{\alpha^{-1}(d-1)}}&=g_{d-1},\\
        u_1x^{-1}_{k_d}u_1^{\delta(k_d)}x_{k_{\alpha^{-1}(d)}}&=g_d.
    \end{cases}
\end{align*}
By rearranging the terms, we get the following system of equations
\begin{align*}
    \begin{cases}
        u_1x^{-1}_{k_d}u_1^{\delta(k_d)}x_{k_{i_1}}&=g_d,\\
        x^{-1}_{k_{i_1}}u_1^{\delta(k_{i_1})}x_{k_{i_2}}&=g_{i_1},\\
        &\,\,\vdots\\
        x^{-1}_{k_{i_{d-1}}}u_1^{\delta(k_{i_{d-1}})}x_{k_{d}}&=g_{i_{d-1}}.
    \end{cases}
\end{align*}
From the equations above we easily obtain
\begin{align*}
    u_1u_1^{x_{k_d}}&=u_1x_{k_d}^{-1}u_1x_{k_d}\\
    &=g_dx^{-1}_{k_{i_1}}u_1^{-\delta(k_d)}u_1x_{k_d}\\
    &=g_dg_{i_1}x^{-1}_{k_{i_2}}u_1^{-\delta(k_{i_1})-\delta(k_d)}u_1x_{k_d}\\
    &=g_dg_{i_1}g_{i_2}x^{-1}_{k_{i_3}}u_1^{-\delta(k_{i_2})-\delta(k_{i_1})-\delta(k_d)}u_1x_{k_d}\\
    &\,\,\,\vdots\\
    &= g_dg_{i_1}\dots g_{i_{d-1}}.
\end{align*}
Note that finding $u$ and $y$ is now reduced to finding $u_1$ and $y_1=x_{k_d}$, which can be done in the same way as we did with $u$ and $y$ because $g_dg_{i_1}\dots g_{i_{d-1}}\in\ker P$. By repeating this process recursively we obtain 
\begin{align*}
    u&=(1,\dots, 1,u_1)\alpha,   &&u_j=(1,\dots, 1,u_{j+1})\alpha_j,\\
    y&=(x_1,\dots, x_d)\beta,  &&y_j=({x_{j}}_1,\dots, {x_{j}}_d)\beta_j,
\end{align*}
for $j\geq 1$. In particular,  $u$ is spherically transitive by \cref{Lemma: Spherically transitive}. This concludes the proof.
\end{proof}

From Lemma \ref{Lemma: Product_of_spherically_transitive} and Theorem \ref{Theorem: (AutT)'=ker(P)} it follows that
$$(\Aut\T)'=\{uv\,|\,u,v \text { sferically transitive}\}=\{[t^a,b]\,|\,a,b\in\Aut\T\}.$$
For the binary tree we have the following stronger result.

\begin{cor}
Let $\T$ be the binary tree. Then every element of $\Aut\T'$ is a commutator of the form $[t^{-1},h]$, where $t$ is the adding machine and $h \in \Aut\T$.
\end{cor}
\begin{proof}
Since $d=2$, the automorphism $u$ constructed in the proof of Theorem \ref{Theorem: (AutT)'=ker(P)} is just the adding machine $t=(1,t)(12)$. It follows that each $g \in (\Aut\T)'$ can be written as $g=tt^x$ for some $x \in \Aut\T$. On the other hand, by Lemma \ref{Lemma: conjucagyinverse} there exists $y \in \Aut\T$ such that $(t^{-1})^y=t$. Thus 
$$
g=t(t^{-1})^{yx}=[t^{-1},yx],
$$
as required.
\end{proof}

\begin{rem}\label{Remark: Description of AutT'}
It follows from \cref{Theorem: (AutT)'=ker(P)} that
\begin{align*}
    (\Aut\T)'&=\{(g_1,\dots, g_d)\sigma \mid g_1\dots g_d\in (\Aut\T)',\,\sigma\in \Alt(d)\}\\
&=\big((\Aut\T)'\cap\st(1)\big)\cdot \Alt(d)
\end{align*}
\end{rem}

\begin{lem}\label{Lemma: coset contains elements of order two}
Let $\T$ be the $d$-adic tree with $d\geq 3$. Then every coset of  $(\Aut\T)'$ in $\Aut\T$ contains an element of order 2.
\end{lem}
\begin{proof}
Consider a sequence $(\varepsilon_0,\varepsilon_1,\varepsilon_2,\dots)\in \prod_{i=0}^\infty \mathbb{Z}_2$ where at least one component is different from 0, and define recursively an element $a \in \Aut\T$ by setting
\begin{align*}
    a=(1,\dots,1,a_1)(12)^{\varepsilon_0}, \quad a_i=(1,\dots,1,a_{i+1})(12)^{\varepsilon_i}.
\end{align*}
Note that  $P(a)=(\varepsilon_0,\varepsilon_1,\varepsilon_2,\dots)$. On the other hand,  using  \cref{Theorem: (AutT)'=ker(P)} and the fact that $P$ is an epimorphism  we have 
\begin{align*}
    \frac{\Aut \T}{\ker P}=\frac{\Aut \T}{(\Aut\T)'}\cong \Imm P = \prod_{i=0}^\infty \mathbb{Z}_2.
\end{align*}
Thus every coset of  $(\Aut\T)'$ in $\Aut\T$ contains an element of the form of $a$. Furthermore, it is easy to see that if $a$ acts on a $d$-adic tree with $d\geq 3$, then $a$ has order 2. Hence every coset of  $(\Aut\T)'$ in $\Aut\T$ contains an element of order 2.
\end{proof}

Let $G$ be any group. Given an element $g \in G$, we will denote by $\langle g\rangle^{G}$  the \textit{normal closure} of $g$ in $G$, and by $\langle \langle g \rangle \rangle$ the \textit{endomorphic closure} of $g$ in $G$, that is
$$
    \langle \langle g \rangle \rangle=\langle \varphi(g):\varphi \in \End(G)\rangle,
$$
 where $\End(G)$ is the monoid of all endomorphisms of  $G$. Clearly $\langle g\rangle^{G}\leq \langle\langle g\rangle\rangle$. It is also obvious that a subgroup $H\leq G$ is fully invariant if and only if $\langle \langle h \rangle \rangle\leq H$, for all $h\in H$.

\begin{lem}\label{Lemma: Circular} Let $\T$ be the $d$-adic tree, and $g,x\in \Aut\T$. If $(x,1,\dots,1)\in\langle\langle g\rangle\rangle$ then $(x,x,\dots,x)\in\langle\langle g\rangle\rangle$.
\end{lem}
\begin{proof} Let consider the map
$$\theta:(x_1,x_2,\dots,x_d)\sigma\in\Aut\T\mapsto(x_2,x_3,\dots,x_d,x_1)\sigma\in\Aut\T.
$$
Then $\theta\in\End(\Aut\T)$ by (\ref{f1}), and the result follows.
\end{proof}

\begin{lem}\label{Lemma: Playing with endomorphic and normal closure of g} Let $\T$ be the $d$-adic tree, and let $g$ be any element of $\in \Aut\T$.
\begin{enumerate}[label=(\roman*)]
\item If $h\in\Aut\T$ and $h\in\langle\langle g\rangle\rangle$, then $\langle\langle h\rangle\rangle\leq\langle\langle g\rangle\rangle$.
\item If $g=(g_1, \dots, g_d) \in \st(1)$, then $\big(g_{\sigma(1)}^{x_{\sigma(1)}}, \dots, g_{\sigma(d)}^{x_{\sigma(d)}}\big)\in \langle g \rangle^{\Aut\T}$ for all $x_i\in \Aut\T$ and $\sigma\in \Sym(d)$. \item If $g=(g_1, \dots, g_d) \in \st(1)$, then $(g_1^{\varepsilon_1}, \dots, g_d^{\varepsilon_d})\in \langle g\rangle^{\Aut\T}$ for all $\varepsilon_i=\pm1$.
\item If $g=(1, \dots, 1,g_i,1,\dots, 1) \in \{1\} \times \dots \times \{1\}\times \Aut\T\times \{1\}\times \dots\times\{1\}$, then 
$$
\langle \langle g \rangle \rangle \geq \langle \langle g_i \rangle \rangle \times \dots \times \langle \langle g_i \rangle \rangle.
$$
%%
%\item If $g = (g_1, \dots, g_d) \in \st(1)$, then for every $h_1 \in \langle \langle g_1 \rangle \rangle$ there exist $h_2, \dots, h_d \in \Aut\T$ such that $(h_1, \dots, h_d)\in \langle \langle g \rangle \rangle$.
\end{enumerate}
\end{lem}
\begin{proof} Since $h\in \langle \langle g \rangle\rangle$, there exist $\varphi_1, \dots, \varphi_s \in \End(\Aut\T)$ such that 
$$h=\varphi_1(g)\cdots\varphi_s(g).$$
Let $x\in\langle \langle h \rangle\rangle$. Then there exist $\psi_1, \dots, \psi_t \in \End(\Aut\T)$ such that 
\begin{align*}
x&=\psi_1(h)\cdots\psi_t(h)\\
&=\psi_1(\varphi_1(g)\cdots\varphi_s(g))\cdots\psi_t(\varphi_1(g)\cdots\varphi_s(g))\\
&=(\varphi_1\psi_1)(g)\cdots(\varphi_s\psi_1)(g)\cdots(\varphi_1\psi_t)(g)\cdots(\varphi_s\psi_t)(g)\in\langle \langle g \rangle\rangle
\end{align*}
since $\varphi_i\psi_j\in\End(\Aut\T)$ for all $i=1,\dots,s$ and $j=1,\dots,t$. Hence (i) holds.
Statement (ii) follows from the fact that if  $x=(x_1, \dots, x_d)\sigma\in \Aut\T$, then by (\ref{f3}) we get
\begin{align*}
g^x=\Big(g_{\sigma(1)}^{x_{\sigma(1)}}, \dots, g_{\sigma(d)}^{x_{\sigma(1)}}\Big).
\end{align*}
%\textcolor{blue}{the "in particular" statement follows from the fact that all elements in $\Aut~\T$ are conjugate of their inverses.}
Statement (iii) is a direct consequence of statement (ii) and Lemma \ref{Lemma: conjucagyinverse}.
In order to prove statement (iv), it obviously suffices to consider the case when $i=1$. For every endomorphism $\varphi$  of $\Aut\T$, define a map $\overline{\varphi}$ by setting
\begin{equation}\label{Equation: How endomorphism work}
\overline{\varphi}((x_1,\dots,x_d)\sigma)=(\varphi(x_1),\dots,\varphi(x_d))\sigma
\end{equation}
for all $(x_1,\dots,x_d)\sigma\in\Aut\T$. It is easy to check that $\overline\varphi$ is an endomorphism of $\Aut\T$. Now if $\varphi \in \End(\Aut\T)$ then $\overline{\varphi}(g)=(\varphi(g_1), 1, \dots, 1)$, hence 
$$
\langle \langle g \rangle \rangle \geq \langle \langle g_1 \rangle \rangle \times \{1\} \times \dots \times \{1\},
$$
and statement (iv) follows by \cref{Lemma: Circular}.
\end{proof}
%Now let $h_1\in \langle \langle g_1 \rangle\rangle$. Then there exist $\varphi_1, \dots, \varphi_s \in \End(\Aut\T)$ such that 
%$$\varphi_1(g_1)\cdots\varphi_s(g_1) =h_1.$$
%For each $i=1, \dots, s$, consider the map $\overline\varphi_i$ as defined in (\ref{Equation: How endomorphism work}). Then we have
%\begin{align*}
%\overline\varphi_1(g)\cdots \overline\varphi_s(g)&=(\varphi_1(g_1), \dots, \varphi_1(g_d))\cdots(\varphi_s(g_1), \dots, \varphi_s(g_d))\\
%&=(\varphi_1(g_1)\dots\varphi_s(g_1), \dots, \varphi_1(g_d)\dots\varphi_s(g_d))\\
%&=(h_1, \dots) \in \langle \langle g \rangle \rangle.    
%\end{align*}
%If $\varphi \in \End(\Aut\T)$, write  $\varphi(g_i)=h_i$ for $i=1, \dots, d$. As a consequence, $\overline\varphi(g)=(\varphi(g_1), \dots, \varphi(g_d))=(h_1, \dots, h_d)$. Then  $(h_1, \dots, h_d)\in \langle \langle g \rangle \rangle$.

\begin{lem}\label{Lemma: Split}
Let $g\in\Aut\T$, and let $\langle \langle g \rangle \rangle \geq\{(x,x,1,\dots,1)\ |\ x\in\Aut\T\}$. Then $\langle \langle g \rangle \rangle \geq(\Aut\T)'\cap\St(1)$.
\end{lem}
\begin{proof}
Our hypothesis, together with \cref{Lemma: Circular} and \cref{Lemma: Playing with endomorphic and normal closure of g} ensures that 
\begin{equation}\label{eq:x,inverse}(x,x^{-1},1,\dots,1)\in\langle\langle g\rangle\rangle
\end{equation}
for all $x\in\Aut\T$.
By \cref{Lemma: Playing with endomorphic and normal closure of g} we also get $(x^y,x,1,\dots,1)\in\langle\langle g\rangle\rangle$ for all $x,y\in \Aut\T$. Hence
\begin{equation}\label{eq:tauconjugate2}
(x^{-1},x^{-1},1,\dots,1)(x^y,x,1,\dots,1)=([x,y],1,\dots,1)\in\langle\langle g\rangle\rangle
\end{equation}
for all $x,y\in \Aut\T$. So by Theorem \ref{Theorem: (AutT)'=ker(P)} we get $(\Aut\T)'\times \{1\} \times \dots \times \{1\} \leq\langle\langle g\rangle\rangle$, and therefore $(\Aut\T)' \times \dots \times (\Aut\T)' \leq \langle\langle g\rangle\rangle$. Let $h=(h_1,\dots,h_d)$ be an element of $(\Aut\T)'\cap\st(1)$. Then we can be write
\begin{align*}
    h&=(h_1,h_1^{-1},1,\dots, 1)(1,h_1h_2,(h_1h_2)^{-1},1,\dots, 1)\cdots\\
    &\qquad\cdots(1,\dots,1,h_1h_2\cdots h_{d-1},(h_1\cdots h_{d-1})^{-1},1)(1,\dots,1,h_1h_2\cdots h_d)\\
    &=(h_1,h_1^{-1},1,\dots, 1)(1,h_1h_2,(h_1h_2)^{-1},1,\dots, 1)\cdots\\
    &\qquad\dots(1,\dots,1,h_1h_2\cdots h_{d-1},(h_1h_2\cdots h_{d-1})^{-1},1)(1,\dots,1,h_1h_2\cdots h_d).
\end{align*}
By (\ref{eq:x,inverse}), the first $d-1$ elements of this product lie in $\langle \langle g \rangle \rangle$. As $h_1\cdots h_d \in \Aut\T'$ by \cref{Theorem: (AutT)'=ker(P)}, it follows by (\ref{eq:tauconjugate2}) and \cref{Lemma: Circular} that $(1,\dots,1,h_1\cdots h_d) \in \langle \langle g \rangle \rangle$. Therefore  $h\in \langle\langle g\rangle\rangle$, as required.
\end{proof}

\begin{lem}\label{Lemma: g not in AutT' then closure g is AutT}
Let $\T$ be the $d$-adic tree, with $d\geq 3$. If $g\notin(\Aut\T)'$, then $\langle \langle g \rangle \rangle = \Aut\T$.
\end{lem}
\begin{proof}
Since $g \notin (\Aut\T)'$, by \cref{Theorem: (AutT)'=ker(P)} we have $P(g)=(\varepsilon_i)_{i\in \mathbb{N}_0}$, with at least one $\varepsilon_i=1$. Let $k\in \mathbb{N}_0$ be the least integer such that $\varepsilon_k=1$, and define $H_g$ to be the set of all elements of $\Aut\T$ whose portrait has even product of permutations at level $k$. Therefore  $H_g$ is a normal subgroup of $\Aut\T$ and $g\not \in H_g$. Furthermore, $|\Aut\T:H_g|=2$. Indeed, since the map $$x(\Aut\T)'\in{\Aut\T}/{(\Aut\T)'}\mapsto P(x)\in\mathbb Z_2^{\mathbb N_0}$$ is an isomorphism, we have
\begin{align*}
    \Aut\T/H_g\cong \frac{\Aut\T/(\Aut\T)'}{H_g/(\Aut\T)'}\cong \frac{\mathbb{Z}_2\times\mathbb{Z}_2\times\mathbb{Z}_2\times\dots}{\{0\}\times\mathbb{Z}_2\times\mathbb{Z}_2\times\dots}\cong\mathbb{Z}_2.
\end{align*}
Now let $a\in \Aut\T$ be an element of order 2, and define, for all $h\in\Aut\T$,  
\begin{align*}
    \varphi_a(h)=\begin{cases}
        1\quad \text{if $h\in H_g$,}\\
        a\quad \text{if $h\not\in H_g$.}
    \end{cases}
\end{align*}
It is easy to see that $\varphi_a\in\End(\Aut\T)$ because $a$ has order 2. Since $g\not \in H_g$, for every element $a\in\Aut\T$ of order 2 we obtain $\varphi_a(g)=a\in\langle\langle g\rangle\rangle$. So the subgroup $\langle\langle g\rangle\rangle$ contains all elements of order 2 of $\Aut \T$. Since $d \geq 3$, by \cref{Lemma: coset contains elements of order two} it also contains a complete system of representatives of the cosets of $(\Aut\T)'$ in $\Aut\T$. Thus it only remains to show that  $(\Aut\T)'\leq \langle\langle g\rangle\rangle$. Note that $\tau=(1,\dots,1)(1\ 2)$ is an element of order 2 and so it is contained in  $\langle\langle g\rangle\rangle$. Therefore for all $h\in \Aut\T$ we get 
\begin{align*}
\tau^{(h^{-1},1,\dots,1)}\tau=(h,h^{-1},1,\dots,1)\in \langle\langle g\rangle\rangle,
\end{align*}
so $(\Aut\T)'\cap \St(1)\leq \langle\langle g\rangle\rangle$ by \cref{Lemma: Split}.
Since
\begin{align*}
    \langle\langle g\rangle\rangle\geq \langle\langle \tau\rangle\rangle\geq \langle\tau\rangle^{\Aut\T}\geq \langle\tau\rangle^{\Sym(d)}=\Sym(d)>\Alt(d),
\end{align*}
by \cref{Remark: Description of AutT'} this yields that $(\Aut\T)'\leq \langle\langle g\rangle\rangle$, concluding the proof.  
\end{proof}

\begin{lem}\label{Lemma: first two coordinates same element and the remaining coordinates trivial}
Let $\T$ be the $d$-adic tree, with $d\geq 3$. If $h=(h_1, h_1, 1, \dots, 1)$ where $h_1 \notin (\Aut\T)'$, then
$
\langle \langle h \rangle \rangle \geq (\Aut\T)' \cap \st(1).
$
\end{lem}
\begin{proof}
    For all $\varphi\in\End(\Aut\T)$, let consider the endomorphism $\overline{\varphi}$ defined in (\ref{Equation: How endomorphism work}). Clearly $\langle\langle h\rangle\rangle\ni\overline{\varphi}(h)=(\varphi(h_1),\varphi(h_1),1,\dots,1)$. Since $h_1 \notin (\Aut\T)'$, by \cref{Lemma: g not in AutT' then closure g is AutT} we have  $\langle\langle h_1\rangle\rangle=\Aut\T$. 
    Hence    $\langle \langle h \rangle \rangle \geq \{(a,a,1,\dots,1) \mid a\in \Aut\T\}$, and the result follows by \cref{Lemma: Split}.
\end{proof}

\begin{lem}\label{Lemma: if g in the derivative of autT but not in the first level}Let $\T$ be the $d$-adic tree with $d\not=2,4$. If $g\in (\Aut\T)'\setminus\st(1)$, then $\langle\langle g\rangle\rangle=(\Aut\T)'$.
\end{lem}
\begin{proof}
It is clear that $\langle\langle g\rangle\rangle\leq (\Aut\T)'$ since $g\in (\Aut\T)'$ and the subgroup $(\Aut\T)'$ is fully invariant. On the other hand, the  map defined by the rule 
\begin{align*}
\varphi\big((h_1,\dots,h_d)\sigma\big)=(1,\dots,1)\sigma
\end{align*}
is an endomorphism of $\Aut\T$ whose kernel is $\St(1)$. Write $g=(g_1,\dots,g_d)\sigma$. Then $\sigma$ is an even permutation and $\langle\langle g\rangle\rangle$ contains the  $\varphi((1,\dots,1)\sigma)$ which may be identified by $\sigma$. Since $d \neq 4$ we have $\langle \sigma \rangle^{\Sym(d)}=\Alt(d)$, and hence
$$
\langle\langle g\rangle\rangle\geq\langle\langle\sigma\rangle\rangle\geq \langle \sigma \rangle^{\Aut\T}\geq \langle \sigma \rangle^{\Sym(d)}=\Alt(d).
$$
In particular, $h=(1,\dots,1)(123)\in \langle\langle g\rangle\rangle$. For all elements $a\in \Aut \T$, set $x=(a^{-1},a,1,\dots,1)$. Then
$h^{x}h^{-1}=(a,a^{-1}, 1, \dots, 1)\in\langle\langle g\rangle\rangle.$ Now \cref{Lemma: Playing with endomorphic and normal closure of g} (iii) yields $(a,a, 1, \dots, 1)\in\langle\langle g\rangle\rangle$ for all $a\in\Aut\T$, and the result follows by \cref{Lemma: Split} and \cref{Remark: Description of AutT'}.
\end{proof}

\section{Verbal subgroups}
\noindent In this section we will determine all verbal subgroups of $\Aut\T$. To this end, we first define inductively the following subgroups:  
\begin{align*}
    M_0&= \Aut\T,\\
    M_1&= (\Aut\T)',\\
    M_2&= (\Aut\T)'\cap \st(1),\\
M_{k+2}&=M_k\times \stackrel{d}{\dots}\times M_k,
\end{align*}
for all $k\geq 1$.
Since $\Aut\T'=\ker P$, it follows that
\begin{align*}
M_3&=(\Aut\T)'\times  \stackrel{d}{\dots}\times(\Aut\T)'\\&=\{(g_1,\dots,g_d)\in \st(1)\ |\ g_i\in (\Aut\T)'\}\leq M_2,
\end{align*}
so it is clear that the subgroups $M_k$ form a decreasing chain. In particular we have  
\begin{align*}
M_{2k+1}&=(\Aut\T)'\times \stackrel{d^k}{\dots}\times(\Aut\T)'\leq \st(k),\\
    M_{2k+2}&=\big((\Aut\T)'\cap \st(1)\big)\times \stackrel{d^k}{\dots}\times \big((\Aut\T)'\cap \st(1)\big)\\
    &=M_{2k+1}\cap \st(k+1).
\end{align*}
Therefore $$\bigcap_{k\geq0}M_k\leq\bigcap_{k\geq1}\St(k)=\{1\}.$$
It is an immediate consequence of \cref{Lemma: g not in AutT' then closure g is AutT} and \cref{Lemma: if g in the derivative of autT but not in the first level} that $M_k$ is a fully invariant subgroup of $\Aut\T$, for all $k\geq0$.
\begin{pr}\label{clauseI}
Let $\T$ be the $d$-adic tree, with $d\geq 3$ and odd, and let  $g=(g_1, \dots ,g_{d^k}) \in M_k \setminus M_{k+1}$. Then $M_k \leq \langle \langle g \rangle \rangle$.
\end{pr}
\begin{proof}
The case $k=0$ is \cref{Lemma: g not in AutT' then closure g is AutT}, while the case $k = 1$ is \cref{Lemma: if g in the derivative of autT but not in the first level}. Let now $k>1$.

First suppose that $g \in M_{2k}\setminus M_{2k+1}$, with $k\geq 1$. Then $g_i\not\in (\Aut\T)'$ for some $i$. Using (ii) of \cref{Lemma: Playing with endomorphic and normal closure of g} we may assume, without loss of generality, that $i = 1$, so  $g_1\not\in \ker P$. Furthermore, we have $g\in \ker P$ and hence $g_1\dots g_{d^k}\in \ker P$. Now since $d$ is odd, there exists some  $j\in\{2,\dots,d^k\}$ such that $P(g_1)\not=P(g_j)$. As before, without loss of generality, we may assume that $j=2$. Therefore $g_1g_2\not\in (\Aut\T)'$. Hence by (iii) and (ii) of \cref{Lemma: Playing with endomorphic and normal closure of g} we obtain $$(g_1,g_2^{-1},g_3,\dots,g_{d^k})(g_2,g_1^{-1},g_3^{-1},\dots,g_{d^k}^{-1})=(g_1g_2,(g_1g_2)^{-1},1,\dots,1) \in \langle \langle g \rangle\rangle,
$$
and so
$
    (g_1g_2,g_1g_2,1,\dots,1)\in \langle\langle g\rangle\rangle.
$
It follows by Lemma \ref{Lemma: first two coordinates same element and the remaining coordinates trivial} that 
\begin{align*}
    \langle\langle g\rangle\rangle\geq ((\Aut\T)'\times\stackrel{d^{k-1}}{\dots}\times(\Aut\T)')\cap\st(k)=M_{2k}.
\end{align*}
Suppose now that $g\in M_{2k+1}\setminus M_{2k+2}$. Then $g_i\in (\Aut\T)'$ for all $i\in\{1,\dots,d^k\}$, and there exists $j\in\{1,\dots,d^k\}$ such that $g_j\not\in \St(1)$. Without loss of generality we may assume that $j=1$. Then, for all $x\in\Aut\T$, using Lemma \ref{Lemma: Playing with endomorphic and normal closure of g} we obtain
\begin{align*}
    (g_1,g_2,\dots,g_{d^k})(g_1^x,g_2^{-1},\dots,g_{d^k}^{-1})=(g_1g_1^{x},1,\stackrel{d^{k}-1}{\dots}, 1)\in \langle\langle g\rangle\rangle.
\end{align*}
Since $d\geq3$ the center of $\Sym(d)$ is trivial. As $g_1 \not\in \St(1)$, there exists some $x\in \Aut\T$ such that $g_1g_1^x\not\in \St(1)$. Now by Lemma \ref{Lemma: if g in the derivative of autT but not in the first level} it follows that
$\langle\langle g_1g_1^x\rangle\rangle \geq (\Aut\T)'$. And finally, by \cref{Lemma: Playing with endomorphic and normal closure of g}, we obtain
$$\langle\langle g\rangle\rangle \geq (\Aut\T)'\times\stackrel{d^k}{\dots}\times(\Aut\T)'=M_{2k+1},$$
concluding the proof.
\end{proof}

\begin{lem}\label{M1=M02}Let $\T$ be the $d$-adic tree. Then $(\Aut\T)^2 =(\Aut\T)'$, in other words $M_0^2=M_1$.
\end{lem}
\begin{proof}
Note that $\Aut\T/(\Aut\T)'$ has exponent 2, since $(\Aut\T)'=\Ker P$ and $\Aut\T/\Ker P\cong \prod \mathbb{Z}/2\mathbb{Z}$. Therefore $(\Aut\T)^2\leq(\Aut\T)'$. On the other hand $(\Aut\T)^2\geq (\Aut\T)'$, because $\Aut\T/(\Aut\T)^2$ has exponent 2, and hence is abelian. 
\end{proof}

\begin{lem}\label{M2=M1ed}Let $\T$ be the $d$-adic tree, with $d\geq 3$ and odd. Then $M_2=M_1^{e_d}$, where $e_d$ is the exponent of the group $\Alt(d)$.
\end{lem}
\begin{proof}
    It is well known that for every odd integer $d\geq3$ there exists an odd integer $o_d$ such that
\begin{align*}
    e_d=\begin{cases}
        2^{k-1}o_d, \quad &\text{if \,}d=2^k+1,\\
        2^ko_d, \quad &\text{if \,}d=2^k+s,\ \text{for some \,}s\in \{2,\dots,2^k-1\}.
    \end{cases}
\end{align*}
By the definition of $e_d$, it is clear that $M_1^{e_d}\leq (\Aut\T)'\cap \St(1)=M_2$. So it remains to prove the reverse inclusion. There are two cases.

If $d=2^k+1$ then  $e_d=2^{k-1}o_d$. Write 
\begin{align*}
    g=(\tau,\underbrace{1,{\dots},1}_{2^{k-1}-1},\tau,\underbrace{1,{\dots},1}_{2^{k-1}-1},1))(1\ 2\ \dots\  2^{k-1})(2^{k-1}+1\ \dots\  2^k).
\end{align*}
Then $g\in M_1=(\Aut~\T)'$ for $\tau=(1,\dots,1)(1\ 2)$. On the other, it easily follows by (\ref{power}) that 
\begin{align*}
    g^{e_d}=(\underbrace{\tau,{\dots},\tau}_{2^k},1)^{o_d}=(\underbrace{\tau,{\dots},\tau}_{2^k},1),
\end{align*}
as $\tau^2 = 1$. Since 
\begin{align*}
 (\tau,\tau,\underbrace{1,{\dots},1}_{2^k-1})=(\tau,1,\underbrace{\tau,{\dots},\tau}_{2^k-1})(1,\tau,\underbrace{\tau, {\dots},\tau}_{2^k-1}),
\end{align*}
it follows by \cref{Lemma: Playing with endomorphic and normal closure of g} (iii) that $(\tau,\tau,1,\stackrel{2^k-1}{\dots},1)\in\langle\langle g\rangle\rangle$. Finally, by \cref{Lemma: first two coordinates same element and the remaining coordinates trivial} and \cref{Lemma: Playing with endomorphic and normal closure of g} (i), we conclude that $M_2=M_1^{e_d}$.

If $d = 2^k + s$ with $s\in \{2,\dots,2^{k}-1\}$, then $e_d = 2^k o_d$. Similarly to the previous case we have that
\begin{align*}
    (\tau,\underbrace{1,{\dots},1}_{2^k-1},\tau,\tau,\underbrace{1,{\dots},1}_{s-3},\tau)(1\ 2\ \dots\ 2^k)(2^k+1\ 2^k+2)\in M_1=(\Aut~\T)' 
\end{align*}
for $\tau=(1,\dots,1)(1\ 2)$. By (\ref{power}) we get
\begin{align*}
    ((\tau,\underbrace{1,{\dots},1}_{2^k-1},\tau,\tau,\underbrace{1,{\dots},1}_{s-3},\tau)(1\ 2\ \dots\ 2^k)(2^k+1\ 2^k+2))^{e_d}
    &=(\underbrace{\tau,{\dots},\tau}_{2^k},\underbrace{1,{\dots},1}_{s})^{o_d}\\&=(\underbrace{\tau,{\dots},\tau}_{2^k},\underbrace{1,{\dots},1}_{s}),
\end{align*}
and again we conclude that $M_2=M_1^{e_d}$.
\end{proof}

\begin{lem}\label{M3=M22}Let $\T$ be the $d$-adic tree, with $d\geq 3$ and odd. Then $M_3=M_2^{2}$.
\end{lem}
\begin{proof}
    Let $g=(g_1, \dots, g_d) \in \St(1)$. Then by \cref{M1=M02} we have
\begin{align*}
    g^2=(g_1,\dots,g_d)^2=(g_1^2,\dots,g_d^2)\in (\Aut\T)'\times\dots\times (\Aut\T)'=M_3.
\end{align*}
Since $M_2\leq \St(1)$, it follows that $M_2^2\leq M_3$. Conversely, by \cref{Lemma: coset contains elements of order two} we know that for each $g\in \Aut\T$ there exist elements $h_g\in\Aut\T$ of order 2 and $g'\in (\Aut\T)'$ such that $g=h_gg'$. Hence
\begin{align*}
    (g^2,1,\dots,1)=(g^2,h_g^2,1,\dots,1)=(g,h_g,1,\dots,1)^2\in M_2^2
\end{align*}
since $(g,h_g,1,\dots,1)\in M_2$ as $gh_g^{-1}=gh_g\in (\Aut\T)'$. 
From \cref{M1=M02} we know that $(\Aut\T)'$ is generated by squares of elements of $\Aut\T$, so we conclude that $M_2^2\geq (\Aut~\T)'\times\stackrel{d}{\dots}\times(\Aut~\T)'=M_3$.
\end{proof}
\begin{lem}\label{M=M}Let $\T$ be the $d$-adic tree, with $d\geq 3$ and odd. Then, for $k\geq 0$,
\begin{align*}
    M_{2k+1}&=M_{2k}^2\\
    M_{2k+2}&=M_{2k+1}^{e_d}.
\end{align*}
\end{lem}
\begin{proof} The case $M_0^2=M_1$ is \cref{M1=M02}. Let assume $k>0$. By using \cref{M3=M22} we get 
\begin{align*}
    M_{2k+1}&=(\Aut\T)'\times\stackrel{d^k}{\dots}\times(\Aut\T)'\\
    &=((\Aut\T)'\times\stackrel{d^k}{\dots}\times(\Aut\T)')\times\stackrel{d^{k-1}}{\dots}\times((\Aut\T)'\times\stackrel{d^k}{\dots}\times(\Aut\T)')\\
    &=(M_3\times\stackrel{d^{k-1}}{\dots}\times M_3)\\
    &=(M_2^2\times\stackrel{d^{k-1}}{\dots}\times M_2^2)\\
    &=(M_2\times\stackrel{d^{k-1}}{\dots}\times M_2)^2\\    &=((\Aut\T)'\cap\St(1)\times\stackrel{d^{k-1}}{\dots}\times (\Aut\T)'\cap\St(1))^2\\
    &=M_{2k}^2.
\end{align*}
Similarly, by \cref{M2=M1ed} we obtain 
\begin{align*}
    M_{2k+2}&=((\Aut\T)'\cap\St(1))\times\stackrel{d^k}{\dots}\times ((\Aut\T)'\cap \St(1))\\
    &=(M_2\times\stackrel{d^k}{\dots}\times M_2)\\
    &=(M_1^{e_d}\times\stackrel{d^k}{\dots}\times M_1^{e_d})\\    
    &=(M_1\times\stackrel{d^k}{\dots}\times M_1)^{e_d}\\
    &=((\Aut\T)'\times\stackrel{d^k}{\dots}\times (\Aut\T)')^{e_d}\\
    &=M_{2k+1}^{e_d}.
\end{align*}
\end{proof}

\begin{thm}
Let $\T$ be the $d$-adic tree, with $d\geq 3$ and odd. Then all fully invariant subgroups of $\Aut\T$ belong to the series $M_0 > M_1 > M_2 >\dots$, and they are verbal.
\end{thm}
\begin{proof}
By \cref{M=M} each $M_i$ is verbal, since it is generated by powers of elements of a verbal subgroup. It remains to prove that every fully invariant subgroup of $\Aut\T$ equals $M_k$ for some $k\geq 0$. Let $G$ be a fully invariant subgroup of $\Aut\T$. Then for all $g\in G$ there exists some $k_g\geq 0$ such that 
\begin{align*} 
    g\in M_{k_g}\setminus M_{k_g+1}.
\end{align*}
Let $\bar k=k_{g_0}$ be the smallest integer in the set $\{k_g:g\in G\}$. Hence $G=M_{\bar k}$ by the definition of $\bar k$. On the other hand $g_0\in M_{_{\bar k}}\setminus M_{_{\bar k+1}}$, so by \cref{clauseI} it follows that 
\begin{align*}
    G\geq \langle\langle g_0\rangle\rangle\geq M_{_{\bar k}},
\end{align*}
concluding the proof.
\end{proof}

\bibliographystyle{amsplain}
\bibliography{bib}
\end{document}